\documentclass[a4j,11pt,leqno]{amsart}
\usepackage{amsmath,amssymb,amscd, amsthm}
\usepackage{epic,eepic,epsfig}
\usepackage{comment} 
\usepackage{mathrsfs, upgreek}
\usepackage[all,cmtip]{xy}
\usepackage{graphicx}
\usepackage{here}
\usepackage{enumerate}  
\usepackage{enumitem}

\usepackage{color}
\makeatletter

\def\@settitle{\begin{center}%
  \baselineskip14\p@\relax
  \normalfont\LARGE\bfseries
  \@title
  \ifx\@subtitle\@empty\else
     \\[1ex] %\uppercasenonmath\@subtitle
     \normalsize\mdseries\@subtitle
  \fi
 \ifx\@didication\@empty\else
     \\[2ex] %\uppercasenonmath\@subtitle
     \large\mdseries\it\@dedication
  \fi
  \end{center}%
}
\def\subtitle#1{\gdef\@subtitle{#1}}
\def\@subtitle{}
\def\dedication#1{\gdef\@dedication{#1}}
\def\@dedication{}
\makeatother

\makeatletter
\def\vmargin@#1#2#3{%% #3 with upper-margin #1, lower-margin #2
\setbox0=\hbox{#3}%
\rule[#1]{0pt}{\ht0}%
\lower\dp0\hbox{\rule[-#2]{0pt}{\dp0}}%
\box0%
}
\def\vmargin#1#2#3{%% #3 with upper-margin #1, lower-margin #2
\mathchoice
{\vmargin@{#1}{#2}{$\displaystyle #3$}}
{\vmargin@{#1}{#2}{$\textstyle #3$}}
{\vmargin@{#1}{#2}{$\scriptstyle #3$}}
{\vmargin@{#1}{#2}{$\scriptscriptstyle #3$}}
}
\makeatother

\usepackage{wrapfig}

\makeatletter
\renewcommand{\section}{\@startsection
{section}{1}{0mm}{5mm}{2mm}{\raggedright\bfseries}}

 \@addtoreset{equation}{section}
\makeatother

\newtheorem{Corollary}[equation]{Corollary}
\newtheorem{Proposition}[equation]{Proposition}
\theoremstyle{definition}
\newtheorem{Definition}[equation]{Definition}
\newtheorem{Example}[equation]{Example}

\newtheorem{Remark}[equation]{Remark}

\newtheorem{Note}[equation]{Note}
\begin{document}
\newcommand{\ig}[1]{\includegraphics[width=4cm]{#1}}
\newcommand{\cSap}[1]{{\cS_{\rm ap}({#1})}}
\newcommand{\tvect}[3]{%
   \ensuremath{\Bigl(\negthinspace\begin{smallmatrix}#1\\#2\\#3\end{smallmatrix}\Bigr)}}
\def\sI{{\mathsf I}}
\def\sJ{{\mathsf J}}
\def\sS{{\mathsf S}}
\def\N{{\mathbb N}}
\def\C{{\mathbb C}}
\def\Z{{\mathbb Z}}
\def\R{{\mathbb R}}
\def\Q{{\mathbb Q}}
\def\cD{{\mathcal{D}}}
\def\cM{{\mathcal{M}}}
\def\cS{{\mathcal{S}}}
\def\cH{{\mathcal{H}}}
\def\cM{{\mathcal{M}}}
\def\cR{{\mathcal{R}}}
\def\cT{{\mathcal{T}}}
\def\bp{{\mathbf p}}
\def\bq{{\mathbf q}}
\def\bP{{\mathbf P}}
\def\bG{{\mathbf G}}
\def\Gal{{\mathrm{Gal}}}
\def\et{\text{\'et}}
\def\ab{\mathrm{ab}}
\def\proP{{\text{pro-}p}}
\def\padic{{p\mathchar`-\mathrm{adic}}}
\def\la{\langle}
\def\ra{\rangle}
\def\scM{\mathscr{M}}
\def\lala{\la\!\la}
\def\rara{\ra\!\ra}
\def\ttx{{\mathtt{x}}}
\def\tty{{\mathtt{y}}}
\def\ttz{{\mathtt{z}}}
\def\bkappa{{\boldsymbol \kappa}}
\def\scLi{{\mathscr{L}i}}
\def\sLL{{\mathsf{L}}}
\def\GL{\mathrm{GL}}
\def\Coker{\mathrm{Coker}}
\def\area{\mathrm{area}}
\def\Ker{\mathrm{Ker}}
\def\CHplus{\underset{\mathsf{CH}}{\oplus}}
\def\check{{\clubsuit}}
\def\kaitobox#1#2#3{\fbox{\rule[#1]{0pt}{#2}\hspace{#3}}\ }
\def\vru{\,\vrule\,}
\newcommand*{\longhookrightarrow}{\ensuremath{\lhook\joinrel\relbar\joinrel\rightarrow}}
\newcommand{\hooklongrightarrow}{\lhook\joinrel\longrightarrow}
\def\nyoroto{{\rightsquigarrow}}
\newcommand{\pathto}[3]{#1\overset{#2}{\dashto} #3}
\newcommand{\pathtoD}[3]{#1\overset{#2}{-\dashto} #3}
\def\dashto{{\,\!\dasharrow\!\,}}
\def\ovec#1{\overrightarrow{#1}}
\def\isom{\,{\overset \sim \to  }\,}
\def\GT{{\widehat{GT}}}
\def\bfeta{{\boldsymbol \eta}}
\def\brho{{\boldsymbol \rho}}
\def\bomega{{\boldsymbol \omega}}
\def\sha{\scalebox{0.6}[0.8]{\rotatebox[origin=c]{-90}{$\exists$}}}
\def\upin{\scalebox{1.0}[1.0]{\rotatebox[origin=c]{90}{$\in$}}}
\def\downin{\scalebox{1.0}[1.0]{\rotatebox[origin=c]{-90}{$\in$}}}
\def\torusA{{\epsfxsize=0.7truecm\epsfbox{torus1.eps}}}
\def\torusB{{\epsfxsize=0.5truecm\epsfbox{torus2.eps}}}
%%%%%%%%%%%%%%%%%%%%%%%%%%%%%%%%%%%%%%%%%%%%%%%%%%%%%%%5

% added by H.Ogawa
\def\bbS{{\mathbb S}}
\def\bbT{{\mathbb T}}
\def\bS{{\mathbf S}}
\def\bT{{\mathbf T}}
\def\idp{{\rm idp}}
%\def\dp{{\rm dp}}

%%%%%%%%%%%%%%%%%%%%%%%%%%%%%%%%%%%%%%%%%%%%%%%%%%%%%%%5

\title{A family of geometric operators on triangles \\
with two complex variables}
%\subtitle{Here can be put a subtitle}  %option
%\dedication{Dedicated to Professor Somebody}  %option
\author{Hiroaki Nakamura}
\address{
Department of Mathematics, 
Graduate School of Science, 
Osaka University, 
Toyonaka, Osaka 560-0043, Japan}
\email{nakamura@math.sci.osaka-u.ac.jp}

\author{Hiroyuki Ogawa}
\address{
Department of Mathematics, 
Graduate School of Science, 
Osaka University, 
Toyonaka, Osaka 560-0043, Japan}
\email{ogawa@math.sci.osaka-u.ac.jp}


\begin{abstract}
Given a plane triangle $\Delta$, one can construct a new triangle $\Delta'$
whose vertices are intersections of two cevian triples of $\Delta$. 
We extend the family of operators $\Delta\mapsto\Delta'$
by complexifying the defining two cevian parameters and study its
rich structure from arithmetic-geometric viewpoints. 
We also find another useful parametrization of the operator family
via finite Fourier analysis and apply it to investigate
area-preserving operators on triangles.
\end{abstract}

\maketitle

%%%%%%%%%%%%%%%%%%%%%%%%%%%%%%%%%%%%%%%%%%%%%%%%%%%%%%%%%%%%%%%%%%%%
% This defines a short-running title of this paper !! 
\markboth{H.Nakamura, H.Ogawa}
{A family of geometric operators on triangles
with two complex variables}
%%%%%%%%%%%%%%%%%%%%%%%%%%%%%%%%%%%%%%%%%%%%%%%%%%%%%%%%%%%%%%%%%%

%\addtocontents{toc}{
%\vspace*{-10mm}
%}
%\tableofcontents
%\footnote[0]{\LaTeX -compiled on: \today}

%\vspace{-3\baselineskip}

\section{Introduction}

In \cite{NO03}, we studied a certain operator $\cS_{p,q}$ on plane 
triangles with two {\it real} parameters $p,q$ ($pq\ne 1, (p,q)\ne(\frac12,\frac12))$
which maps any triangle $\Delta ABC$ to a new triangle $\Delta A'B'C'$
such that the vertex $A'$ (resp. $B'$, $C'$) is obtained as the intersection
of the $(p:1-p)$-cevian from the vertex $B$ 
(resp. $C$, $A$) and 
the $(1-q:q)$-cevian from the vertex $A$
(resp. $B$, $C$).

The aim of this paper is to extend $\cS_{p,q}$ `algebraically' for {\it complex} parameters
$p$, $q$ and to study basic properties of the total collection of those operators. 
Although our extension differs from
geometrically obvious one via generalized cevians,
we still can give geometrical interpretation 
of $\cS_{p,q}$ for complex $p,q$ (in Remark \ref{rem2.6}).

After establishing setup and basic properties in \S2-3, 
we introduce in \S 4,
a {\it moduli parameter} $\xi_{p,q}$ which controls the effect of $\cS_{p,q}$ 
on the shapes (similarity classes) of triangles.
In \S 5, it turns out that finite Fourier analysis enables us to find natural 
separation of numerator and denominator of $\xi_{p,q}$ in the form
$\xi_{p,q}=\eta_{p,q}/\eta_{p,q}'$, which leads us to introduce 
a more natural family
of operators $\cS[\eta,\eta']$ defined all over the pairs $(\eta,\eta')\in \C^2$.
We give explicit expression of the original parameters $p$, $q$ as rational functions
in $\eta,\eta'$. Finally in \S 6, 
we study ``area-preserving normal operators'' parametrized by
a standard torus $S^1\times S^1$ in $\C^2$

Throughout this paper, we use the notations: 
%$i:=e^{2 \pi i/4}$, 
$\rho:=e^{2 \pi i/6}$, $\omega:=e^{2 \pi i/3}$.

\section{Operators $\cS_{p,q}$}

A {\it triangle} is a set $\{a,b,c\}$ of three points (vertices) 
on the complex plane $\C$. 
We will consider any ordered triple $(a,b,c)\in\C^3$ with distinct coordinates
as a {\it triangle triple} representing $\{a,b,c\}$.
It is said to be {\it degenerate} if $a,b,c$ are collinear. 
A non-degenerate triangle triple $(a,b,c)$ is said to be  {\it positive}
(resp. {\it negative})  
%representing $\Delta=\{a,b,c\}$ if $(a-b)/(c-b)$ belongs to the
if it represents a triangle $\{a,b,c\}$ with $(a-b)/(c-b)$ 
belonging to the
upper half plane $\cH^+=\{z\in\C\mid Im(z)>0\}$
(resp. the lower half plane $\cH^-=\{z\in\C\mid Im(z)<0\}$).

In this paper, we study operations $\cS_{p,q}$ ($p,q\in\C$, $pq\ne 1$)
on the triangle triples defined by
\begin{equation} \label{eq1.1}
\cS_{p,q}(a,b,c)=(a',b',c'):
\begin{cases}
a'=\hspace{-4mm}& \ \alpha_{p,q}\, a + \beta_{p,q}\, b +\gamma_{p,q} \, c ;\\
b'=\hspace{-4mm}& \ \alpha_{p,q}\, b + \beta_{p,q}\, c +\gamma_{p,q} \, a ;\\
c'=\hspace{-4mm}& \ \alpha_{p,q}\, c + \beta_{p,q}\, a +\gamma_{p,q} \, b,
\end{cases}
\end{equation}
where,  
$$
\alpha_{p,q}=\frac{p(1-q)}{1-pq},
\  
\beta_{p,q}=\frac{q(1-p)}{1-pq},
\
\gamma_{p,q}=\frac{(1-p)(1-q)}{1-pq}.
$$
%\begin{Note}
We will often identify $\cS_{p,q}$ with its obvious extension to the
linear operator on $\C^3$ defined by the same formula (\ref{eq1.1}). 
The above operator $\cS_{p,q}$ is designed to preserve centroids of triangles
due to the condition $\alpha_{p,q}+\beta_{p,q}+\gamma_{p,q}=1$
(cf.\,Proposition \ref{cor3.4}).

For real parameters $(p,q)$, the operator $\cS_{p,q}$ was studied 
in \cite{NO03} where $a',b',c'$ are given as the 
intersection points of certain cevians 
of $\{a,b,c\}$ determined by $(p,q)$. 
In \cite{Ka15}, T.Kanesaka
inspected $\cS_{p,q}$ extended to complex variables $(p,q)$ 
in the above form.
In Remark \ref{rem2.6}, we will discuss a geometrical interpretation 
of $\cS_{p,q}$ for complex $p,q$ which uses more tools 
than cevians.

%\end{Note} 
%
%

\begin{Proposition} \label{prop1.1}
Let $p,q\in\C$ satisfy $pq\ne 1$. The following two conditions are equivalent.

\begin{enumerate}[label=(\roman*),font=\upshape]
\item
$\cS_{p,q}$ maps every positive triangle triple to a positive triangle triple. 
\item
$p=q\ne \frac{1}{2}$ or $\frac{(p-1)(2q-1)}{p-q}\in\R\cup\cH^+$.
%\setminus\{\rho\}$.
\end{enumerate}
\end{Proposition}

\begin{proof}
Given a positive triangle triple $(a,b,c)$, let 
$\cS_{p,q}(a,b,c)=(a',b',c')$. 
For $z=(a-b)/(c-b)\in \cH^+$, one computes: 
$$
F(z)=\frac{a'-b'}{c'-b'}=\frac{(1-q)(2p-1)z+(1-p)(1-2q)}{(1-p)(2q-1)z+(p-q)}.
$$
The condition (i) is equivalent to that the mapping $z\mapsto F(z)$ defines 
a linear fractional transformation sending $\cH^+$ into $\cH^+$.
When $p=q$, $(p,q)\ne (\frac{1}{2},\frac{1}{2})$ rephrases
non-degeneracy of $\{a',b',c'\}$, in which case
$F(z)=\frac{z-1}{z}$ maps $\cH^+$ onto $\cH^+$ certainly.
When $p\ne q$, writing $t:=(p-1)(2q-1)/(p-q)$, one finds
$$
F(z)=\frac{(t-1)z-t}{tz-1}
$$
which represents a non-degenerate linear fractional map 
with  $F(\rho)=\rho$ if and only if $t\ne \rho$.
In this case, the condition (i) is equivalent to $F(0)=t\in\R\cup\cH^+$. When $t=\rho$, constantly $F(z)=\rho$ 
so that $\cS_{p,q}$ maps every positive triangle triple 
to a positive equilateral triangle triple. 
\end{proof}

\begin{Definition} \label{def1.2}
The operator $\cS_{p,q}$ will be called {\it regular} 
if the condition of Proposition \ref{prop1.1}
is satisfied. 
Moreover, we shall call $\cS_{p,q}$ to be {\it normal}, if 
it is regular with 
$(p-1)(2q-1)/(p-q)\in \R$.
\end{Definition}

\begin{Example}\label{classical_nest}
Let $\Delta=(a,b,c)$ denote a triangle triple.
In \cite{NO03}, we presented two major examples $\cS_{0,q}$ and $\cS_{1-q,q}$ for $q\in\R$.
$\cS_{0,q}(\Delta)$ is a triangle triple whose vertices are the ``$(1-q):q$''-division points 
of sides of $\Delta$, while $\cS_{1-q,q}(\Delta)$ is a triangle triple obtained from 
the ``$(1-q):q$''-cevians. In \cite{H09}, these are studied under the name of 
$s$-medial and $s$-Routh triangles $\cM_s(\Delta)$, $\cR_s(\Delta)$ 
which, in our notations, correspond to $\cS_{0,1-s}(\Delta)$ and $\cS_{s,1-s}(\Delta)$ respectively.
There is another interesting construction called the $s$-median triangle of $\Delta$
(written $\cH_s(\Delta)$ in \cite{H09}).
In \cite{NO2}, we discuss a generalization of $\cH_s(\Delta)$ from the viewpoint
of the present paper. 
\end{Example}

\begin{Example} \label{NapoleonTh}
The famous Napoleon's theorem tells that, for any triangle $\Delta$, 
the centers of the three equilateral triangles sharing one of the sides of $\Delta$
form an equilateral triangle. 
This means, for $q={\frac{1-\omega^2}{3}}=\frac{1}{2}+\frac{\sqrt{3}i}{6} $, 
$\cS_{0,q}(\Delta)$ is equilateral for 
every triangle triple $\Delta$. 
This $\cS_{0,q}$ gives a regular but a non-normal operator which motivates us
to investigate complex parameters $(p,q)$.
\end{Example}

\section{Matrix analysis of $\cS_{p,q}$}

In this section, we summarize basic facts on $\cS_{p,q}$ from the 3 by 3 matrix point of view.
Let 
$$
\sI:=\begin{pmatrix} 1 & 0 & 0 \\ 0 & 1& 0\\ 0 & 0 & 1 \end{pmatrix}, \quad
\sJ:=\begin{pmatrix} 0 & 1& 0 \\ 0 & 0 & 1 \\ 1 & 0 & 0  \end{pmatrix}
$$
and define 
$$
\sS_{p,q}:=\alpha_{p,q}\sI + \beta_{p,q} \sJ +\gamma_{p,q} \sJ^2.
$$
{}From (\ref{eq1.1}) follows that  
\begin{equation}
\cS_{p,q}(a,b,c)=\sS_{p,q} \tvect{a}{b}{c}
\end{equation}
where triangle triples $(a,b,c)$ is regarded as column vectors in $\C^3$.
It is natural to identify $\cS_{p,q}$ with $\sS_{p,q}\in \GL_3(\C)$
acting on the column vector space $\C^3$ on the left.
It is easy to see that
\begin{equation} \label{cyclic_perm}
\sS_{0,0}=\sJ^2;\quad
\sS_{1,q}=\sI \quad (q\ne 1); \quad 
\sS_{p,1}=\sJ \quad (p\ne 1)
\end{equation}
and that $\sS_{1,1}$ is essentially undefined ($p=q=1$ is a point of indeterminacy).
Observe that the operators $\cS_{0,0}$, $\cS_{1,q}$ $(q\ne 1)$ induce 
cyclic permutations on the vertices of triangle triples.
\begin{Definition}
We call $\cS_{0,0}$, $\cS_{p,1}$ $(p\ne 1)$ {\it cyclic permutation operators},
and call $\cS_{1,q}$ $(q\ne 1)$ {\it identity operators}.
\end{Definition}

\begin{Example}
On the other hand, there are no $\cS_{p,q}$ that universally induce transpositions 
on vertices of triangle triples. However, given an individual triangle triple $(a,b,c)$,
one finds 
\begin{align*}
\cS_{p,q}(a,b,c)& =(a,c,b) \quad\text{for  } (p,q)=\left(\frac{c-a}{b-a},\frac{a-c}{b-c}\right), \\
\cS_{p,q}(a,b,c)&=(b,a,c) \quad\text{for  }  (p,q)=\left(\frac{b-c}{a-c},\frac{c-b}{a-b}\right), \\
\cS_{p,q}(a,b,c)&=(c,b,a) \quad\text{for  }  (p,q)=\left(\frac{a-b}{c-b},\frac{b-a}{c-a}\right).
\end{align*} 
These are called {\it reflection operators} for the triangle triple $\Delta=(a,b,c)$ and form 
first basic examples of non-regular $\cS_{p,q}$.
It is a good exercise to show that $\cS_{p,q}$ is a reflection operator for some $\Delta$
if and only if 
the parameters $p$, $q\in\C$ satisfy 
$pq\neq1$ and $pq=p+q$. 
(Hint: We have $\cS_{p,q}(a,b,c)=(a,c,b)$ if $a=(1-q)c+qb$.)
\end{Example}

Below we first prepare a notion of weighted mean (average) of 
two matrices (also of two vectors), 
which will be used in the subsequent Remark.

\begin{Definition}
Given a complex number $r\in\C$ and two matrices $A,B$ of a same size, 
we define the {\it $r$-weighted mean}, 
or for short, {\it $r$-average} of $A$ and $B$, by
$\mu_r(A,B):=(1-r)A+rB$.
\end{Definition}

\begin{Remark} \label{rem2.6}
({\it Geometrical interpretation}.)
In \cite[p.76, Remark]{Nic13}, G.Nicollier gave a useful 
geometrical interpretation of the operator $\cS_{p,q}$ 
in terms of convolution with generalized Kiepert triangle
(\cite[p.69]{Nic13}. See also proof of Proposition \ref{prop4.1} below.)

Here, we give another interpretation in terms of 
consecutive operations of weighted arithmetic means of triangles.
It follows from a simple matrix calculation that
\begin{equation}
\sS_{p,q}=\mu_{\frac{1-q}{1-pq}}(\sJ, \mu_p(\sJ^2,\sI))
=\mu_{\frac{1-p}{1-pq}}(\sI, \mu_q(\sJ^2,\sJ)).
\end{equation}
Geometrical interpretation of the above identity applied to a triangle triple 
$\Delta=(a,b,c)$ may be phrased as follows.
First, form the $p$-average $\Delta'=\mu_p(\sJ^2\Delta,\Delta)$ 
of `label-permuted' triangles $\sJ^2\Delta=(c,a,b)$ and $\Delta=(a,b,c)$.
Then, one obtains $\cS_{p,q}(\Delta)$ as the $\frac{1-q}{1-pq}$-average 
of $\sJ\Delta=(b,c,a)$ and $\Delta'$.
In another way, form the $q$-average $\Delta''=\mu_q(\sJ^2\Delta,\sJ\Delta)$ 
of $\sJ^2\Delta=(c,a,b)$ and $\sJ\Delta=(b,c,a)$. 
Then, $\cS_{p,q}(\Delta)$ is obtained as the $\frac{1-p}{1-pq}$-average 
of $\Delta=(a,b,c)$ and $\Delta''$.

\end{Remark}

\begin{Note}  \label{note2.7}
Generically, it holds that
$$
\sS_{p_2,q_2}\sS_{p_1,q_1}=\sS_{p_1,q_1}\sS_{p_2,q_2}=\sS_{p,q}
$$
with 
$$
p=\frac{\lambda_2+2\lambda_3}{\lambda_1}, \qquad q=\frac{\lambda_1-2\lambda_3}{\lambda_2},
$$
where
\begin{align*}
\lambda_1&=1-(1-2(1-p_1)(1-p_2))(1-2(1-q_1)(1-q_2)), \\
\lambda_2&=1-2(1-p_1 p_2)-(1-2(1-p_1)q_2)(1-2(1-p_2)q_1), \\ 
\lambda_3&=(1-p_1q_1)(1-p_2 q_2). 
\end{align*}
But the collection $\bbS_{p,q}:=\{\sS_{p,q}\mid p,q\in\C, pq\ne 1\}$ is not 
closed under multiplication in $\GL_3(\C)$. 
For example, $(p_1,q_1)=(\frac13,\frac14)$, $(p_2,q_2)=(-\frac78,-\frac19)$
implies $\lambda_1=\lambda_2+2\lambda_3=0$ which shows essential 
indeterminacy of $p$ in this special case. 
Meanwhile, if  $(p_1,q_1)=(\frac13,\frac14)$, $(p_2,q_2)=(-\frac27,\frac15)$, then
$\lambda_1=0$ while $\lambda_2+2\lambda_3\ne 0$. 
In this case, one would be able to regard $p=\infty$ with 
suitable interpretation about $\cS_{p_1,q_1}\cS_{p_2,q_2}=\cS_{\infty,q}$.
Later in \S 5, we will remedy this inconvenience by extending the set
$\bbS_{p,q}$ to a multiplicative submonoid of $M_3(\C)$. 
Cf. Remark \ref{rem4.9}.
\end{Note}

%%%%%%%%%%%%%%%%%%%%%%%%%%%%%%%%%%%%%%%%%%%%%%%%%%%%%%%%%%%%%%

\section{Actions on the moduli disk $\cD$}

In \cite{NO03}, we introduced a moduli space of the similarity
classes of triangles as 
the Poincar\'e disk $\cD=\{z\in\C;|z|<1\}$ 
by associating with any triangle $\Delta$
the well-defined modulus
\begin{equation} \label{eq2.1}
\phi(\Delta):=\left(
\frac{a+b\omega+c\omega^2} {a+b\omega^2+c\omega}
\right)^3
\in\cD,
\end{equation}
when $\Delta=\{a,b,c\}$ is labeled by a positive triangle triple $(a,b,c)$.
Now, let us introduce two quantities $t_{p,q}$ and $\xi_{p,q}$ by 
\begin{equation}\label{eq2.2}
t_{p,q}:=\frac{(p-1)(2q-1)}{p-q},
\quad \xi_{p,q}:=\frac{1+t_{p,q}\omega}{\,1+t_{p,q}\omega^{2}}=
{\frac{(p-q)+(p-1)(2q-1)\,\omega}{(p-q)+(p-1)(2q-1)\,\omega^2}}
.
\end{equation}
It is not difficult to see that $\cS_{p,q}$ is regular (Definition \ref{def1.2})
if and only if $|\xi_{p,q}|\le 1$.
Indeed, each regular operation $\cS_{p,q}$
induces a moduli action $T_{p,q}:\cD\to\cD$ satisfying 
$T_{p,q}(\phi(\Delta))=\phi(\cS_{p,q}(\Delta))$ 
for all triangles $\Delta$ given by the formula
\begin{equation} \label{eq2.3}
T_{p,q}(z)=\xi_{p,q}^3\!\cdot\! z \ .
\end{equation}
Thus, the moduli action $T_{p,q}$ induced from a regular operator $\cS_{p,q}$ 
maps the moduli disk
$\cD=\{z\in\C;|z|<1\}$ into a smaller or the same disk
$\{z\in\C;|z|<|\xi_{p,q}|^3\}$.
In particular, a normal operator $\cS_{p,q}$ gives rise to $T_{p,q}$
that rotates $\cD$ with angle $Arg(\xi_{p,q}^3)$ in the anti-clockwise way.
It is also reasonable to call such $T_{p,q}$ {\it regular} 
(resp. {\it normal}) when
$Im(t_{p,q})\ge 0\Leftrightarrow |\xi_{p,q}|\le 1$
(resp. $Im(t_{p,q})= 0\Leftrightarrow |\xi_{p,q}|= 1$).

Here, we summarize basic properties of $\cS_{p,q}$, some of which are consequences of 
(\ref{eq2.3}).
Write $\overline{\Delta}$ for the complex conjugate of a triangle triple $\Delta$.

\begin{Proposition} 
\label{cor3.4}
\begin{enumerate}[label=(\roman*),font=\upshape]
\item
If $\cS_{p,q}(\Delta)={\Delta'}$ then $\cS_{\bar p,\bar q}(\overline{\Delta})=\overline{\Delta'}$, 
and vice versa holds.
\item
The operator $\cS_{p,q}$ preserves centroids of triangles, and 
maps equilateral triangle triples to 
equilateral triangle triples.
\item
Let $\Delta$ be an arbitrary triangle triple, and $r,q\in\C$, $r\ne 2$, $q\ne -1$.
Then,  all of $\cS_{\frac{1}{2},r}(\Delta)$, $\cS_{r,\frac{1}{2}}(\Delta)$ 
and $\cS_{q,q}(\Delta)$ are similar to $\Delta$. 
\end{enumerate}
\end{Proposition}

\begin{proof} 
(i) follows immediately from the definition of $\cS_{p,q}$.
(ii) The first assertion is easy to see from $a+b+c=a'+b'+c'$ in (\ref{eq1.1})
due to the identity $\alpha_{p,q}+\beta_{p,q}+\gamma_{p,q}=1$ encoded in
our definition. 
For the second, suppose first that $\Delta$ is given as a positive triangle triple 
representing an equilateral triangle. 
The action of $\cS_{p,q}$ on the similarity class of $\Delta$ is reduced to 
the moduli operator $T_{p,q}$ acting on the moduli disc through multiplication by $\xi_{p,q}^3$.
In particular, it fixes 0 which is the modulus of the equilateral triangles.
If $\Delta$ is given as a negative equilateral triple, then 
$\overline{\Delta}$ is a positive equilateral
triple, which, by the above argument, is sent 
to an equilateral triple by $\cS_{\bar p,\bar q}$.
Then, apply (i) to get the asserted conclusion.
(iii) Again by virtue of (i), it suffices to show the assertion for positive triangle triples.
Direct computation shows that $\xi_{\frac{1}{2},r}^3=\xi_{r,\frac{1}{2}}^3=\xi_{q,q}^3=1$.
It follows from (\ref{eq2.3}) that the corresponding operators 
$\cS_{\frac{1}{2},r}(\Delta)$, $\cS_{r,\frac{1}{2}}(\Delta)$
and $\cS_{q,q}(\Delta)$
do not change the similarity classes
of triangles.
\end{proof}

Observe in (\ref{eq2.2}), the quantities $t_{p,q}$ and $\xi_{p,q}$
make senses to be valued in $\bP^1(\C)=\C\cup\{\infty\}$ as long as 
$(p,q)\ne(1,1),(\frac12,\frac12)$ (even when $pq=1$ with $\cS_{p,q}$ undefined).
Therefore, at the level of moduli action, we shall extend $T_{p,q}$
for all complex pairs $(p,q)\ne(1,1),(\frac12,\frac12)$.
Consider now the bijective mapping 
$$
\psi:\bP^1_t(\C)\isom\bP^1_\xi(\C): t\mapsto \xi=\frac{1+t\omega}{1+t\omega^2}
$$ 
of the (Riemann) $t$-sphere onto the $\xi$-sphere modelled on (\ref{eq2.2}), which
in particular maps $\psi(0)=1$, $\psi(\rho)=0$ and $\psi(\rho^{-1})=\infty$.
Recall here the notations $\rho:=e^{2 \pi i/6}$, $\omega:=e^{2 \pi i/3}$.
\begin{Definition}
Let $T(\C):=\bP^1_t(\C)-\{\rho,\rho^{-1}\}$ and 
introduce a commutative group structure $[+]$ by pulling back the
multiplication on $\bG_m(\C)=\bP^1_\xi(\C)-\{0,\infty\}$ via $\psi$.
\end{Definition}

\begin{Proposition} 
Notations being as above, the following assertions hold:
\begin{enumerate}[label=(\roman*),font=\upshape]
\item
For $t,t'\in T(\C)$, $t[+]t'=\frac{t+t'-tt'}{1-tt'}$. %modified by H.Ogawa
\item
If $[N]:T(\C)\to T(\C)$ denotes the multiplication by $N$
in the sense of $[+]$ for $N\in\Z$, then 
$$
[N]\cdot t=
\frac{(1+\omega t)^N-(1+\omega^{-1}t)^N}{\rho(1+\omega t)^N-\rho^{-1}(1+\omega^{-1}t)^N}.
$$
\item
The $N$-division points of $T(\C)$ are given by
$$
t=\frac{\sin \pi \frac{k}{N}}{\sin\pi(\frac{k}{N}+\frac{1}{3})}
\quad
(k=0,1,...,N) .
$$
\end{enumerate}
\end{Proposition}

\begin{proof}
These formulae follow in straightforward ways from the above 
definitions after noting
$\psi(t[+]t')=\psi(t)\psi(t')$, $\psi([N]t)=\psi(t)^N$ and
`$\psi([N]t)=1\Leftrightarrow$ $t$ is an $N$-division point'.
\end{proof}

The above group structure on $T$ is essentially same as what
was introduced in
Komatsu \cite{Ko04} and Ogawa \cite{Oga03} for a twisted Kummer theory
designed to control $n$-cyclic extensions of number fields containing $\cos(2\pi i/n)$.
(Our case is a special case of $n=6$ with $t^{-1}$ corresponding to the 
standard coordinate in their theory). 
B\'enyi-\'Curgus \cite{BC12} studied the same additive structure
on the `normal' locus $T(\R)$.

In fact, $T$ has a natural structure of algebraic torus defined over $\Q$:
one can equip $T$ with the $\Q$-structure of the affine conic (ellipse) 
$S: u^2+uv+v^2=1$ by transformations by 
$$
u = \frac{1-t^2}{t^2-t+1}, v = \frac{t(t-2)}{t^2-t+1}; \quad
t=-\frac{1+u}{v}.
$$
Then, the additive structure $[+]$ on $T$ can be understood as that 
induced from a
comultiplication on the structure ring 
$\Q[u,v]/(u^2+uv+v^2-1)$ of $S$:
See N.Suwa \cite{Su08} for more thorough 
generalization of Kummer theory toward contexts of 
commutative group schemes.

\begin{Remark}
If $p\ne q$ and $t=(p-1)(2q-1)/(p-q)$ lies in the {\it lower} half plane 
$\cH^{-}$, then 
$\cS_{p,q}$ alters some positive triangle triples $(a,b,c)$ into non-positive ones, 
i.e., into triples $(a',b',c')$ with $Im( \frac{a'-b'}{c'-b'})<0$. 
This case still makes sense and is interesting, however, in this paper 
we do not focus on non-regular operations.
\end{Remark}

\begin{Remark}
The above expression (\ref{eq2.1}) of our triangle modulus $\phi(\Delta)$
was notified by G.Nicollier to the first author. 
It is a crucial observation that the triangles on the same radius in $\cD$
have the same Brocard angle $\upomega(\Delta)$, i.e., 
$$
|\phi(\Delta)|=r^3 \ (0<r<1) \Leftrightarrow 
\frac{\cot(\upomega(\Delta))}{\sqrt{3}}=\frac{1+r^2}{1-r^2}.
$$  
See the Introduction of \cite{Nic13}. 
The quantity inside the cube in (\ref{eq2.1}) is called a shape function
on triangle triples and has been intensively studied by recent works 
by M.Hajja et. al. (cf.\,e.g., \cite{H09} and references therein).
\end{Remark}

\section{Fourier transform $\cS[\eta,\eta']$}

Let $\area(\Delta)$ denote the area of a triangle $\Delta$.
We begin this section with:

\begin{Proposition} \label{prop4.1}
Suppose $p,q\in\R$ $(pq\ne 1)$.
The ratio $\area(\cS_{p,q}(\Delta))/\area(\Delta)$ is given by
$$
\frac{\area(\cS_{p,q}(\Delta))}{\area(\Delta)}=
\left|\frac{p-q}{1-pq}\cdot (1+t_{p,q}\omega^2)\right|^2,
$$
where the quantity $t_{p,q}$ is as in (\ref{eq2.2}).
\end{Proposition}

\begin{proof}
Since the operation $\cS_{p,q}$ belongs to affine geometry, 
the ratio $\area(\cS_{p,q}(\Delta))/\area(\Delta)$ is independent
of choice of $\Delta$. 
We may assume, for example, $\Delta$
is the equilateral triangle $\Delta_0=(1,\omega,\omega^2)$. 
We make use of G.Nicollier's interpretation of $\cS_{p,q}(\Delta)$
in \cite{Nic13}\,\S 4\,(Remark in p.76-77): if $p(2q-3)\ne -1$ then
$\cS_{p,q}(\Delta)\equiv \frac{p(2q-3)+1}{1-pq}\Delta\ast
K(\frac{q-p}{p(2q-3)+1})$, where $\equiv$ means congruence
up to parallel translations and $\Delta\ast K(z)$ is 
(the convolution with the Kiepert triangle $K(z)$)
the triangle obtained from the three apices 
of `$(1-z:z)$-ears' of the sides of $\Delta$.
%Since $\area(\Delta_0\ast K(z))/\area(\Delta_0)=3|z-(\frac12-\frac{\sqrt{3}i}{6})|^2$,
It follows easily that 
\begin{equation}
\frac{\area(\Delta_0\ast K(z))}{\area(\Delta_0)}
=3\left|
z-\left(
\frac12-\frac{\sqrt{3}i}{6}
\right)
\right|^2.
\end{equation}
Noting this and the equality $p(2q-3)+1=(2-t_{p,q})(q-p)$, we obtain
the desired result by computing the left hand side as
$$
\frac{\area(\cS_{p,q}(\Delta_0))}{\area(\Delta_0)}=
\left|\frac{(2-t_{p,q})(p-q)}{(1-pq)}\right|^2\cdot 3
\left|\frac{1}{2-t_{p,q}}-\left(\frac12-\frac{\sqrt{3}i}{6}\right)
\right|^2
=\left|\frac{p-q}{1-pq} (1+t_{p,q}\omega^2)\right|^2.
$$
We may remove the assumption $p(2q-3)\ne -1$ for
the ratio is obviously continuous in parameters $p,q$.
\end{proof}

\begin{Example}
The famous Routh's theorem states that 
both $\cS_{\frac{1}{3},\frac23}(\Delta)$ and  $\cS_{\frac{2}{3},\frac13}(\Delta)$
have one-seventh area of $\Delta$. 
Applying Proposition \ref{symm_pq} (ii) shown later, 
we can express matrices for these two Routh operators as  
$\sS_{\frac{1}{3},\frac23}=\sJ\, \cS_{\frac45,\frac23}=\sJ^2 \,\sS_{\frac{1}{3},\frac15}$
and 
$\cS_{\frac{2}{3},\frac13}=\sJ\, \cS_{\frac{1}{5},\frac13}=\sJ^2\, \cS_{\frac23,\frac45}$.
The twofold (families of) Routh triangles 
$\cS_{\frac{1}{3},\frac23}(\Delta)$ and $\cS_{\frac{2}{3},\frac13}(\Delta)$
share the centroid with $\Delta$, however, they 
are {\it generally not} point-symmetrical to each other.
Using Proposition \ref{symm_pq} (iii), we find that $\cS_{p,q}(\Delta) $ and 
$\cS_{q,p}(\Delta)$ form a point-symmetrical pair for every triangle $\Delta$
if and only if $4pq-3p-3q+2=0$.
For example, $(\cS_{\frac13,\frac35}, \cS_{\frac35,\frac13})$,
$(\cS_{\frac25,\frac47}, \cS_{\frac47,\frac25})$ are such pairs, but
$(\cS_{\frac13,\frac23}, \cS_{\frac23,\frac13})$ is not.
In \cite[Example 2.15]{NO2}, we give another characterization of 
the Routh operators by their special relationships to 
``generalized median operators''.
The above area formula for real parameters $p,q$ 
in Proposition \ref{prop4.1}
may be regarded as a special case of a more general
formula given in \cite{BC13}, where other interesting numerical 
examples with rational parameters can be found.
Generalization for complex parameters will be discussed in 
Proposition \ref{GeneralAreaFormula} below.
\end{Example}

Motivated by Proposition \ref{prop4.1}, 
let us introduce new quantities $\eta_{p,q}$ {and $\eta'_{p,q}$} by:
\begin{equation} \label{eq4.2}
\eta_{p,q}:=(1+t_{p,q}\omega)\frac{p-q}{1-pq}, \quad 
{\eta'_{p,q}:=(1+t_{p,q}\omega^2)\frac{p-q}{1-pq}}.
\end{equation}

It is clear that the parameter $\xi_{p,q}$ introduced in (\ref{eq2.2})
satisfies $\xi_{p,q}=\eta_{p,q}/\eta'_{p,q}$.
Recall that the parameter $\xi_{p,q}$
comes from the action 
$T_{p,q}$ on the moduli disk $\cD$, the moduli disk
for the similarity classes of triangle.
%Since ${|\eta'_{p,q}|^2=}|\eta_{p,q}\xi_{p,q}^{-1}|^2$ represents 
%the area ratio $\frac{\area(\cS_{p,q}(\Delta))}{\area(\Delta)}$
%of Proposition \ref{prop4.1}, the triple
%$(\xi_{p,q},\eta_{p,q},\eta'_{p,q})$ 
%contains information of $\cS_{p,q}$ about its
%scale and shape change of triangles. 
The above numerator-denominator separation $(\eta_{p,q},\eta'_{p,q})$ 
of the moduli action $\xi_{p,q}$ does, in fact, recover 
the operator $\cS_{p,q}$ itself as follows:

\begin{Proposition} \label{prop.abc}
Notations being (\ref{eq1.1}), the operator $\cS_{p,q}$ is determined by the
parameter $(\eta_{p,q},\eta'_{p,q})$ with
\begin{align*}
\alpha_{p,q}&=\frac{1}{3}(1+\eta_{p,q}+\eta'_{p,q}), \\ 
\beta_{p,q}&=\frac{1}{3}(1+\omega\,\eta_{p,q}+\omega^2\eta'_{p,q}), \\ 
\gamma_{p,q}&=\frac{1}{3}(1+\omega^2\eta_{p,q}+\omega\,\eta'_{p,q}). 
\end{align*}
\end{Proposition}

\begin{proof}
This follows immediately after
substituting $p=p(\eta,\eta')$, $q=q(\eta,\eta')$ for $\alpha_{p,q}$, $\beta_{p,q}$ and $\gamma_{p,q}$
in  (\ref{eq1.1}). 
\end{proof}

Now, we shall give interpretation of the parameter $(\eta_{p,q},\eta'_{p,q})$
from the view point of finite Fourier series. In the spirit of 
\cite{Sch50}, \cite{Nic13}, considering a triangle triple $\Delta=(a,b,c)$ is equivalent
to considering its Fourier transform $(\psi_0,\psi_1,\psi_2)$ defined as the coefficients
of a quadratic polynomial 
$\Psi_\Delta(T)=\psi_0(\Delta)+\psi_1(\Delta) T+\psi_2(\Delta) T^2\in\C[T]$
uniquely determined by the identities
$\Psi_\Delta(1)=a, \Psi_\Delta(\omega)=b, \Psi_\Delta(\omega^2)=c$.
Explicitly, they are given by
$\psi_0(\Delta)=\frac{1}{3}(a+b+c)$,
$\psi_1(\Delta)=\frac{1}{3}(a+b\omega^2+c\omega)$,
$\psi_2(\Delta)=\frac{1}{3}(a+b\omega+c\omega^2)$.

\begin{Definition} 
Let $\Delta=(a,b,c)$ be a triangle triple. Notations being as above, 
the Fourier transformed vector $\uppsi(\Delta)$ of $\Delta$ is defined by
$$
\uppsi(\Delta)=\tvect{\psi_0(\Delta)}{\psi_1(\Delta)}{\psi_2(\Delta)}.
$$
\end{Definition}

By simple (matrix) computation, we find that the parameter $(\eta_{p,q},\eta'_{p,q})$
is just the ratio of the 2nd and 3rd Fourier coefficients.
More precisely
\begin{Proposition} \label{DiagActionSpq}
For each triangle triple $\Delta$, it holds that
\begin{equation*}
\uppsi(\cS_{p,q}(\Delta))=\mathrm{diag}(1,\eta'_{p,q},\eta_{p,q})\cdot
\uppsi(\Delta).
\end{equation*}
\end{Proposition}

\begin{proof}
The asserted identity is equivalent to
\begin{equation}
\label{defSyy}
\sS_{p,q}=W
\begin{pmatrix}
1 & 0 & 0 \\
0 & \eta'_{p,q} & 0 \\
0 & 0 & \eta_{p,q} 
\end{pmatrix}
W^{-1}
\quad
\text{with} 
\quad
W=\begin{pmatrix}
1 & 1 & 1 \\
1 & \omega & \omega^2 \\
1 & \omega^2 & \omega
\end{pmatrix}
\end{equation}
in the space of $3\times 3$-matrices.
\end{proof}

Observing the above simple form of $\cS_{p,q}$ acting on the Fourier transformed triangle triples, 
we are led to the idea of viewing the operation $\cS_{p,q}$ by the parameters
$(\eta,\eta')\in\C^2$ rather than the original parameter $(p,q)\in\C^2$ 
with $pq\ne 1$:

\begin{Definition} \label{defSeta}
For $(\eta,\eta')\in\C^2$, we define the operator
$\cS[\eta,\eta']$ on the triangle triples by the action on their
Fourier transform vectors, and define its corresponding matrix
$\sS[\eta,\eta']$ as
\begin{equation} \label{eq4.9}
\uppsi(\cS[\eta,\eta'](\Delta))=\begin{pmatrix}
1 & 0 & 0 \\
0 & \eta' & 0 \\
0 & 0 & \eta 
\end{pmatrix}
\uppsi(\Delta), \quad
\sS[\eta,\eta']:=
W\begin{pmatrix}
1 & 0 & 0 \\
0 & \eta' & 0 \\
0 & 0 & \eta 
\end{pmatrix}
W^{-1},
\end{equation}
where $W$ is the Vandermonde matrix used in (\ref{defSyy}).
\end{Definition}

\begin{Remark} \label{rem4.9}
It is now evident that the collection 
$$
\bbS:=\{\sS[\eta,\eta']\mid \eta,\eta'\in\C\}
=\{\alpha\sI+\beta\sJ+\gamma\sJ^2\mid \alpha+\beta+\gamma=1\}
$$
forms a multiplicative submonoid of $M_3(\C)$ 
containing the set $\bbS_{p,q}$ considered in Note \ref{note2.7}.
Let $\bbS^\times:=\{\sS[\eta,\eta']\mid \eta,\eta'\in\C^\times\}$ be
the set of invertible elements of $\bbS$ which forms an abelian subgroup
isomorphic to $(\C^\times)^2$ in $\GL_3(\C)$.
Note that $\bbS_{p,q}\subset \bbS\supset \bbS^\times$. 
In the special case when $(p_1,q_1)=(\frac13,\frac14)$, $(p_2,q_2)=(-\frac78,-\frac19)$,
$\lambda_1=\lambda_2+2\lambda_3=0$
discussed in loc.\,cit. we can take the corresponding parameters 
$(\xi_1,\eta_1,\eta_1')=(-\frac{1+4\omega}{3+4\omega}, \frac{1+4\omega}{11},-\frac{3+4\omega}{11})$,
$(\xi_2,\eta_2,\eta_2')=(\frac{3\omega-1}{3\omega^2-1}, \frac{11}{13}(3\omega-1),
\frac{11}{13}(3\omega^2-1))
$. This shows that $\sS_{p_1,q_1}$, $\sS_{p_2,q_2}\in\bbS^\times\cap \bbS_{p,q}$.
Accordingly, the composition $\sS_{p_1,q_1}\sS_{p_2,q_2}=\sS[\eta_1,\eta_1']\sS[\eta_2,\eta_2']=
\sS[\eta_1\eta_2,\eta_1'\eta_2']=\sS[\omega^2,\omega]$ 
lies in $\bbS^\times$, whereas
it is not available in $\bbS_{p,q}$ as discussed in Note \ref{note2.7}. 
\end{Remark}

We now generalize Proposition \ref{prop4.1} for arbitrary complex parameters.
\begin{Proposition} \label{GeneralAreaFormula}
Let $\Delta=(a,b,c)$ be a triangle triple with its Fourier transform
$\uppsi(\Delta)=(\psi_0,\psi_1,\psi_2)$, and let $(\eta,\eta')\in \C^2$.
Then, 
\begin{enumerate}[label=(\roman*),font=\upshape]
\item
$\area(\Delta)=\frac{3\sqrt{3}}{4} \Bigl| |\psi_1|^2-|\psi_2|^2 \Bigr|$.

\item
If $\Delta$ is non-degenerate, then
$$
\frac{\area(\cS[\eta,\eta'](\Delta))}{\area(\Delta)}=
\left|
\frac{|\eta'\psi_1|^2-|\eta\psi_2|^2}{|\psi_1|^2-|\psi_2|^2}
\right|.
$$
\item The operator $\cS[\eta,\eta']$ preserves
areas of triangles if and only if $|\eta|=|\eta'|=1$.
\end{enumerate}
\end{Proposition}

\begin{proof}
(i). Recall that $\area(\Delta)=\bigl|\mathrm{Im}((b-a)\overline{(c-a)}\bigr|/2$.
In terms of $\uppsi(\Delta)$, one can express 
$$
(b-a)\overline{(c-a)}=(\omega-1)^2|\psi_1|^2+(\bar\omega-1)^2|\psi_2|^2
+(\omega-1)(\bar\omega-1)(\psi_1\overline{\psi_2}+\psi_2\overline{\psi_1}).
$$
Noting that the last term is real and that 
$\mathrm{Im}((\omega-1)^2)=-3\sqrt{3}/2$, 
we obtain the assertion.
(ii) follows from (i) and Definition \ref{defSeta}.
(iii) First, consider special cases of positive and negative equilateral triangles 
$\Delta_0=(1,\omega,\omega^2)$ and $\Delta_1=(1,\omega^2,\omega)$,
where $\uppsi(\Delta_0)=(0,1,0)$, $\uppsi(\Delta_1)=(0,0,1)$.
By virtue of (ii), 
$\area((\cS[\eta,\eta'])(\Delta_i))=\area(\Delta_i)$ 
for $i=0,1$ imply $|\eta'|=1$, $|\eta|=1$ respectively.
Conversely, if $|\eta|=|\eta'|=1$, then it is easy to see from (ii)
that $\cS[\eta,\eta'] $ preserves areas of every triangle.
\end{proof}

Before going further, let us quickly summarize how 
the parameters $p,q$ can be recovered from
{$\eta=\eta_{p,q}$, $\eta'=\eta'_{p,q}$}:

\begin{Proposition} \label{prop4.5}
Define two rational functions $p(y,y')$, $q(y,y')\in\C(y,y')$ by
\begin{equation*}
p(y,y')=\frac{1+y+y'}{2-\omega y-\omega^2\,y'}, \quad
q(y,y')=\frac{1+\omega\,y+\omega^2\,y'}{2-y-y'}.
\end{equation*}
Then, for $\eta=\eta_{p,q}$ {and $\eta'=\eta'_{p,q}$}, it holds that
$$
\cS_{p,q}=\cS_{p(\eta,\eta'),q(\eta,\eta')}=\cS[\eta,\eta']
$$ 
as long as involved quantities are well-defined. 
Moreover, if $p_i=p(\eta_i,\eta_i')$, $q_i=q(\eta_i,\eta_i')$ $(i=1,2,3)$ 
are given with $\eta_1\eta_2=\eta_3$, $\eta_1'\eta_2'=\eta_3'$, then
$\cS_{p_1,q_1}\circ \cS_{p_2,q_2}=\cS_{p_3,q_3}$ and vice versa.
\end{Proposition}

\begin{proof}
The first statement follows by straightforward computations
solving the equations
(\ref{eq2.2}) and (\ref{eq4.2}) reversely.
The latter statement can be seen from the above formula (\ref{eq4.9}).
Or, more directly, it follows from the fact that the cyclic matrices 
$\sS_{p,q}\in \C\sI+\C\sJ+\C\sJ^2$ are simultaneously 
diagonalized to $\mathrm{diag}(1,\eta'_{p,q},\eta_{p,q})$
 (via base change by the Vandermonde matrix $(\omega^{i-1,j-1})_{i,j}$).
\end{proof}

The following matrix properties are often useful to interpret 
relations between operators $\cS_{p,q}$ geometrically.

\begin{Proposition} \label{symm_pq}
Below, quantities $p,q\in\C$ are understood to be taken 
so that involved rational functions and operators are defined.
\begin{enumerate}[label=(\roman*),font=\upshape]
\item
$\sS_{p,q}=\sS[\eta,\eta']$ if and only if 
$\sS_{q,p}=\sS[\eta' \omega^2,\eta\omega]=\sJ\,\sS[\eta',\eta]$.
\item $\sS_{p,q}=\sJ \,\sS_{p_1,q_1}=\sJ^2\, \sS_{p_2,q_2}$, \\
where $(p_1,q_1)=\left(\frac{q(p-1)}{2pq-p-q}, 1-p\right)$, 
$(p_2,q_2)=\left(1-q, \frac{p(q-1)}{2pq-p-q}\right)$.
\item $\sS_{p,q}+\sS_{p',q'}=\frac23(\sI+\sJ+\sJ^2)$,
where $(p',q')=\left(\frac{2-3p+pq}{1+3q-4pq}, \frac{2-3q+pq}{1+3p-4pq}\right)$.
\end{enumerate}
\end{Proposition}

\begin{proof}
(i) follows from the observation in Proposition \ref{prop4.5}
that $p(y,y')$ and $q(y,y')$ are interchanged to each other when
$(y,y')$ is replaced by $(\omega^2 y', \omega y)$.
The last equality 
\begin{equation} \label{NO2.1}
\sS[\eta' \omega^2,\eta\omega]=\sJ\,\sS[\eta',\eta]
\end{equation}
follows by a simple matrix computation (cf.\,\cite[(2.1)]{NO2}).
(ii) is just a translation of (\ref{NO2.1}) in terms of the original parameters $(p,q)$. 
(iii): The variable $(p',q')$ is defined so that 
$\sS_{p,q}=\sS[\eta,\eta']$ and $\sS_{p',q'}=\sS[-\eta,-\eta']$.
The identity reflects a simple consequence of 
$\frac12(\sS[\eta,\eta']+\sS[-\eta,-\eta'])=W\,\mathrm{diag}(1,0,0)\,W^{-1}
=\frac13(\sI+\sJ+\sJ^2)$.
\end{proof}

\begin{Example}
In the following table, we summarize examples of $\cS_{p,q}$ with special 
parameters $p,q, \xi_{p,q}, \eta_{p,q}$ and $\eta_{p,q}'$: 
{\small 
$$
\begin{array}{ll@{\quad}|@{\quad}lll@{\quad}|@{\quad}l}
 p & q & \xi_{p,q} & \eta_{p,q} & \eta'_{p,q} & \text{$\cS_{p,q}$} \\
 \hline
 1 & \ast(\ne 1) & 1 & 1 & 1 &  (\ref{cyclic_perm}): \text{\tiny identity} \rule{0pt}{10pt}\\
 \ast(\ne 1) & 1 & \omega & \omega^2 & \omega &  (\ref{cyclic_perm}):\text{\tiny cyclic permutation} \\
 0 & 0    & \omega^2 & \omega & \omega^2 &  (\ref{cyclic_perm}):\text{\tiny cyclic permutation} \\
 \hline 
 0 & 1-s  & {\frac{s+(1-s)\omega}{s\omega+(1-s)}} 
& s\omega+(1-s)\omega^2 
& s\omega^2+(1-s)\omega 
& \text{Ex.\ref{classical_nest}}: \text{\it $s$-medial} \text{ (\cite{H09},\cite{NO03})} 
\rule{0pt}{12pt}  \\[5pt]
 s & 1-s  & {\frac{\omega+s}{1+\omega\,s}} & {\frac{1-2s}{1-s+s^2}}(\omega^2+\omega\,s) 
 & {\frac{1-2s}{1-s+s^2}}(\omega+\omega^2\,s) 
 & \text{Ex.\ref{classical_nest}}: \text{\it $s$-Routh} \text{ (\cite{H09},\cite{NO03})} \\[5pt]
 {\frac{2s}{s+3}} & {\frac{s}{2s-3}} & {\frac{s+\omega}{s+\omega^2}} & s+\omega & s+\omega^2 
 &\text{Ex.\ref{classical_nest}}:  \text{\it $s$-median} \text{ (\cite{H09},\cite{NO2})} \\[2pt]
 \hline
 0 & {\frac{1-\omega^2}{3}} & 0 & 0 & -1 & \text{Ex.\ref{NapoleonTh}}:\text{\it Napoleon}  
 \rule[0pt]{0pt}{13pt}
 \\[5pt]
 0 & {\frac{1-\omega}{3}} & \infty & -1 & 0 & \text{Ex.\ref{NapoleonTh}}: \text{\it Napoleon (reversed)}
\\[2pt]  
\hline
 r (\ne 2) & {\frac{1}{2}} & 1 & {\frac{1-2r}{r-2}} & {\frac{1-2r}{r-2}} & \text{Prop.\,\ref{cor3.4} (iii)} 
 \rule{0pt}{12pt} \\
 {\frac{1}{2}} & r(\ne 2) & \omega & {\frac{1-2r}{r-2}\,\omega^2} 
& {\frac{1-2r}{r-2}\,\omega} &\rule{0pt}{12pt}\text{Prop.\,\ref{cor3.4} (iii)} \\[5pt]
q(\ne -1) & q(\ne -1) & \omega^2 & {\frac{1-2q}{1+q}\,\omega} & {\frac{1-2q}{1+q}\,\omega^2} & \text{Prop.\,\ref{cor3.4} (iii)} \\
\hline
 {\frac{1+2e^{i\,\theta}}{2+e^{i\,\theta}}} & {\frac{1}{2}} & 1 & { e^{i \theta}} & {e^{i \theta}} & \text{$\theta$-rotation}\rule{0pt}{12pt} \\[5pt]
\end{array}
$$
}
\end{Example}

%%%%%%%%%%%%%%%%%%%%%%%%%%%%%%%%%%%%%%%%%%%%%%%%%%%%%%%%%%%

\section{Area preserving operators $\cSap{\theta_x,\theta_y,\theta_y'}$}

We once again recall the notations $\rho:=e^{2 \pi i/6}$, $\omega:=e^{2 \pi i/3}$.
Given $\cS_{p,q}$, we have the secondary parameters
$(\xi_{p,q},\eta_{p,q},\eta_{p,q}')$ with $\eta_{p,q}=\xi_{p,q}\eta_{p,q}'$ as
studied in the preceding sections.
If $|\xi_{p,q}|=1$, then it operates on the moduli disc of similarity classes
by a bijective rotation, and 
if $|\eta_{p,q}|=|\eta'_{p,q}|=1$, then $\area(\cS_{p,q}(\Delta))=\area(\Delta)$ 
for every triangle triple $\Delta$
(cf.\,(\ref{eq2.3}), Proposition \ref{GeneralAreaFormula}).

We begin with:

\begin{Proposition}
The operator $\cS_{p,q}$ has a finite period $N$ if and only if
$\xi_{p,q}^N=\eta_{p,q}^N={\eta'_{p,q}\!\!{}^N}=1$, 
In particular, such $\cS_{p,q}$ is normal and preserves areas of triangles.
\end{Proposition}

\begin{proof} This follows immediately from our discussions above and 
in \S 5.
\end{proof}

\begin{Definition}
We call $\cS_{p,q}$ 
an {\it area-preserving normal operator} 
if the associated parameters satisfy 
$|\xi_{p,q}|=|\eta_{p,q}|=|\eta'_{p,q}|=1$.
\end{Definition}

In Proposition \ref{prop4.5}, we considered two rational functions 
$p(y,y')$, $q(y,y')\in\C(y,y')$ that represent the parameter 
$(p,q)$ by the new parameter $(\eta, \eta')$.
This is natural because $\eta,\eta'$ are eigenvalues of $\cS_{p,q}$. 
Meanwhile, when we restrict our attention to the area-preserving operators,
we find the parameter $(\xi=\eta/\eta',\eta)$ still deserves to be utilized.
So we shall define two rational functions $\bp(x,y)$, $\bq(x,y)$ by
\begin{align}
\bp(x,y)&:=p(y,\frac{y}{x})=\rho \cdot\frac{xy+x+y}{xy+2\rho x+\rho^2y},   \\
\bq(x,y)&:=q(y,\frac{y}{x})= \rho^{-1}\cdot \frac{xy-\rho x-\rho^{-1}y}{xy-2x+y}.
\end{align}
In fact, we find the following algebraic functional equations 
\begin{align}
&\bp(x,y)=\bq(\omega x^{-1}, \omega^2 x^{-1}y), \quad
\bq(x,y)=\bp(\omega x^{-1}, \omega^2 x^{-1}y);    \label{eq5.5} \\
&\bp(y,x)=\frac{\bp(x,y)(\bq(x,y)-1)}{\rho(\bp(x,y)-1)\bq(x,y)-(\bp(x,y)-\bq(x,y))};  \label{eq4.11}\\
&\bq(y,x)=\frac{\bq(x,y)-1}{(1+\rho)\bq(x,y)-1}. \label{eq4.12}
\end{align}
Moreover, on the {\it normal area-preserving torus}
\begin{equation}
\cT:=S^1\times S^1=\{(x,y)\in\C; |x|=|y|=1\},
\end{equation}
we have a closer look at the functions 
$\bp(x,y)$ and $\bq(x,y)$ with their extra-symmetry properties:
namely, if we consider $\bp(x,y)$ and $\bq(x,y)$ as $\C$-valued partial 
functions on $\cT$, then,
\begin{enumerate}[label=(\roman*),font=\upshape]
\item $\bp(x,y)$ is defined on $\cT-\{(\omega,\omega^2)\}$.
\item $\bq(x,y)$ is defined on $\cT-\{(1,1)\}$. 
\end{enumerate}
On these regions of definition, they have functional equations:
\begin{equation}
\bp(x,y)+\overline{\bp(\omega^2x^{-1},\omega y^{-1})}=1,\quad
\bq(x,y)+\overline{\bq(x^{-1},y^{-1})}=1.
\label{eq4.10}
\end{equation}
The above functional equations 
(\ref{eq5.5})-(\ref{eq4.10})
can be verified by 
straightforward calculations (assisted by symbolic algebraic system, once they
are found).

One possible way to see a source reason behind the above symmetric 
formulas may be connecting $\bq(x,y)$ with
the following simple rational function
\begin{equation}
R(x,y):=\frac{2xy-x-y}{\sqrt{3}(y-x)}.
\end{equation}
In fact, one can interpret a modified function
\begin{equation}
Q(x,y):=i\bigl(\bq(x,y)-\frac{1}{2}\bigr)
\quad\text{   as to be:   }\quad
Q(x,y)=\frac{1}{2}\left(\frac{\sqrt{3}R(x,y)-1}{R(x,y)+\sqrt{3}}\right).
\end{equation}
The identity (\ref{eq4.12}) and the first formula of (\ref{eq4.10}) are consequences 
of the obvious symmetry: 
\begin{equation} \label{eqR}
R(x,y)+R(y,x)=0, \quad R(x^{-1},y^{-1})=\overline{R(x,y)}
\qquad (|x|=|y|=1),
\end{equation}
while (\ref{eq4.11}) and
the second formula of  (\ref{eq4.10})
are their translations through (\ref{eq5.5}).

\begin{Corollary}
The values of $\bq(x,y)$ on $\cT$ are determined from
those values $\bq(e^{i\theta},e^{i\phi})$ on the region 
$\cT_0:=\{(\theta,\phi)
\mid 0<\theta<2\pi,\, 0\le\phi\le\theta,\,
\theta+\phi\le 2\pi\}$.
\end{Corollary}

Note that the identity (\ref{eq5.5}) tells how the function 
$\bq(x,y)$ determines $\bp(x,y)$.

\begin{proof}
This follows from observing values of $R(x,y)$ and 
the symmetric relations (\ref{eqR}):
The values of $R(e^{i\theta},e^{i\phi})$ are determined on 
$\cT_0$. Note that on the boundary of $\cT_0$, 
the value of $R(x,y)$ is undetermined only at $\theta=\phi=0$;
elsewhere it takes
$\frac{i}{\sqrt{3}}\tan\frac{\theta}{2}$ 
on $\{(\theta,\phi)\mid 0<\phi=2\pi-\theta\le \pi\}$,
constant $1/\sqrt{3}$ on 
$\{(\theta,0)\mid 0<\theta<2\pi\}$
and $\infty$ on 
$\{(\theta,\phi)\mid 0<\theta=\phi\le\pi\}$.
\end{proof}

At the end of this article, let us give several illustrations on normal 
area-preserving operations on triangles. 
For convenience, let us identify the normal area-preserving torus as:
$$
\cT\cong 
\{(\theta_x, \theta_y, \theta_y')\in(\R/\Z)^3\mid \theta_x=\theta_y-\theta_y'\}.
$$
via $(x,y)\leftrightarrow (e^{2\pi i\theta_x},e^{2\pi i\theta_y})$.

\begin{Definition}
Define
$$
\cSap{\theta_x,\theta_y,\theta_y'}:=\cS[e^{2\pi i\theta_y},e^{2\pi i\theta_y'}].
$$ 
\end{Definition}

%\newpage
\begin{Example}
Here are two orbits of area-preserving operators 
$\cSap{\theta_x,\theta_y,\theta_y-\theta_x}$ of order 20 and 28
starting from the initial triangle $\Delta=(0,\,1,\,0.7+0.8\,i)$: 

\begin{figure}[h]%[htbp]
\begin{center}
\begin{tabular}{cc}
 \ig{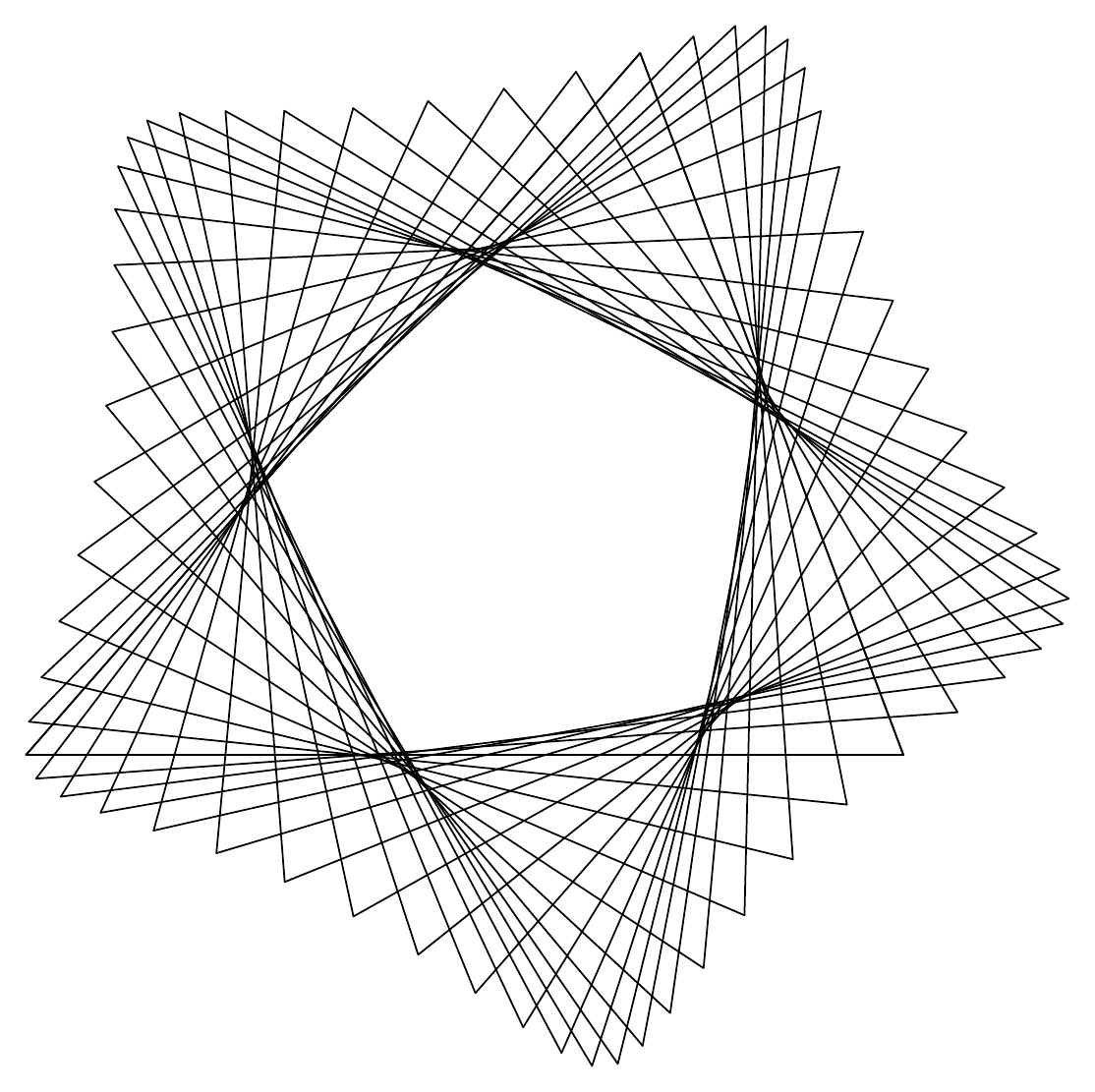} & 
 \ig{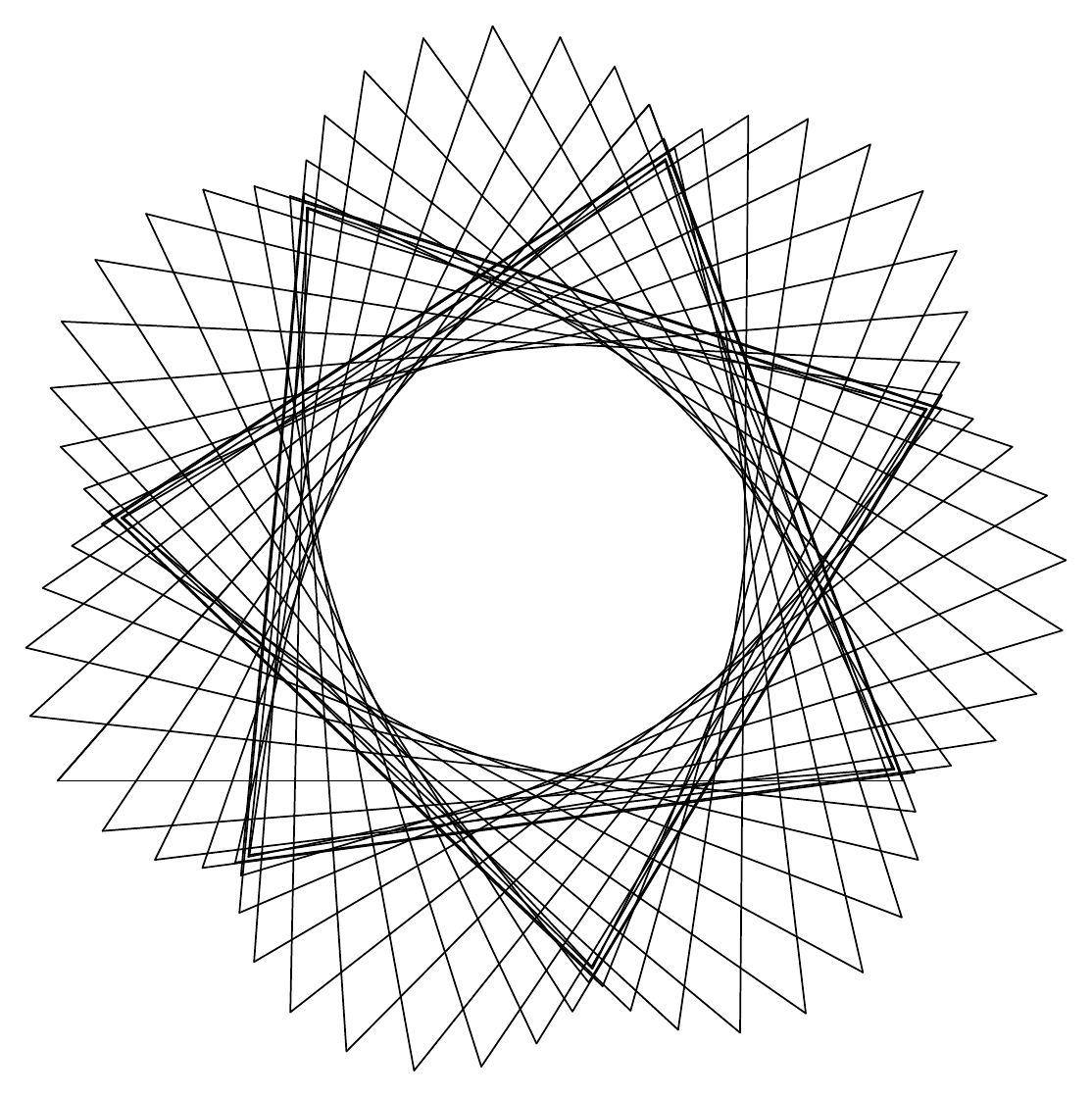} 
\\
 \text{$ \{\cSap{\frac{1}{4},\frac{1}{5},\frac{1}{5}-\frac{1}{4}}^n(\Delta)\}_{n\in\Z}  $} & 
 \text{$\{ \cSap{\frac{1}{4},\frac{1}{7},\frac{1}{7}-\frac{1}{4}}^n(\Delta)\}_{n\in\Z}$} 
\end{tabular}
\end{center}
\end{figure}
\end{Example}

%
%\pagebreak
\begin{Example} 
Simple rotations through the angle $\theta$ are realized by $\cSap{0,\theta,\theta}$.
Here are two examples starting from $\Delta=(0,\,1,\,0.7+0.3\,i)$:
\begin{figure}[h]%[htbp]
\begin{center}
\begin{tabular}{cc}
\ig{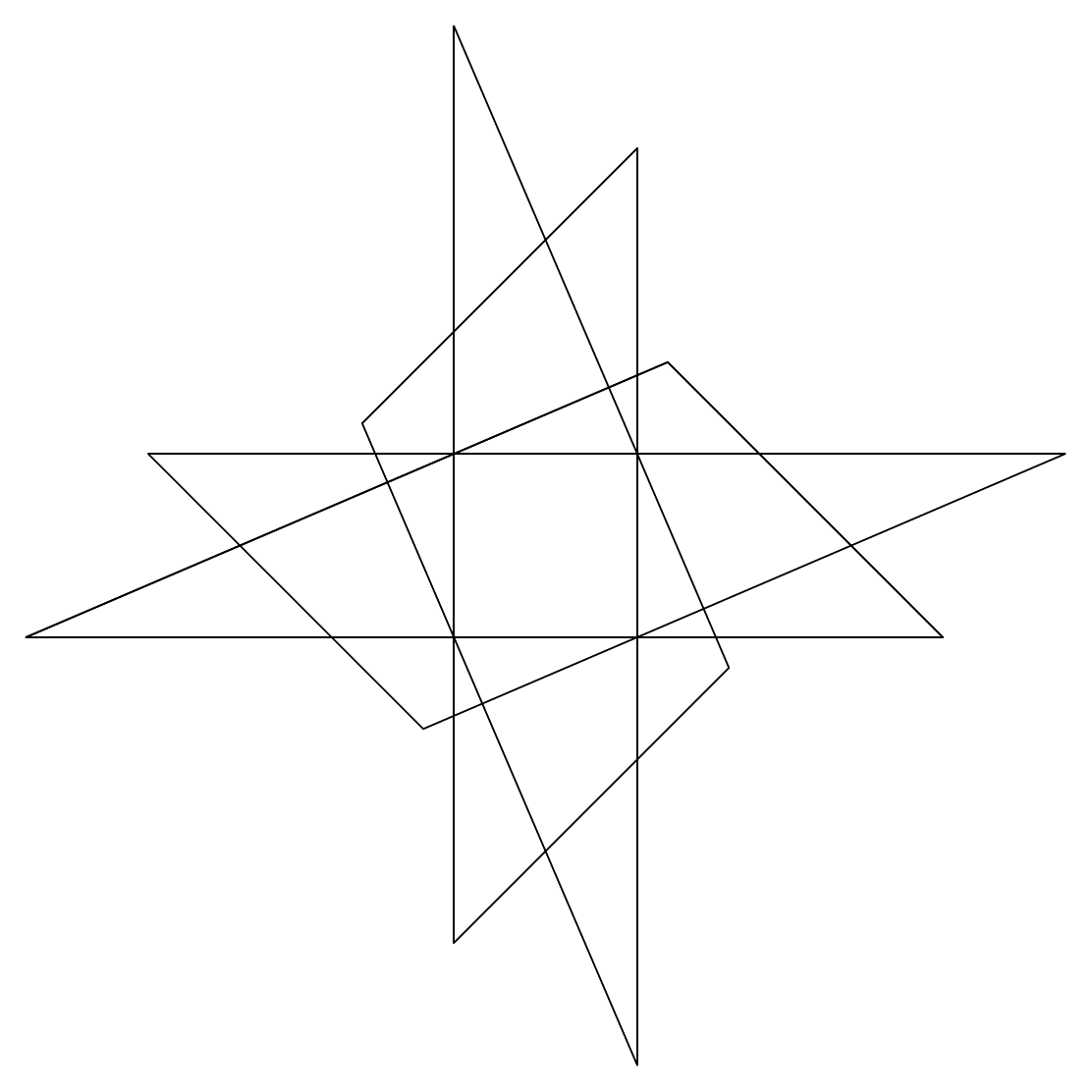} & 
 \ig{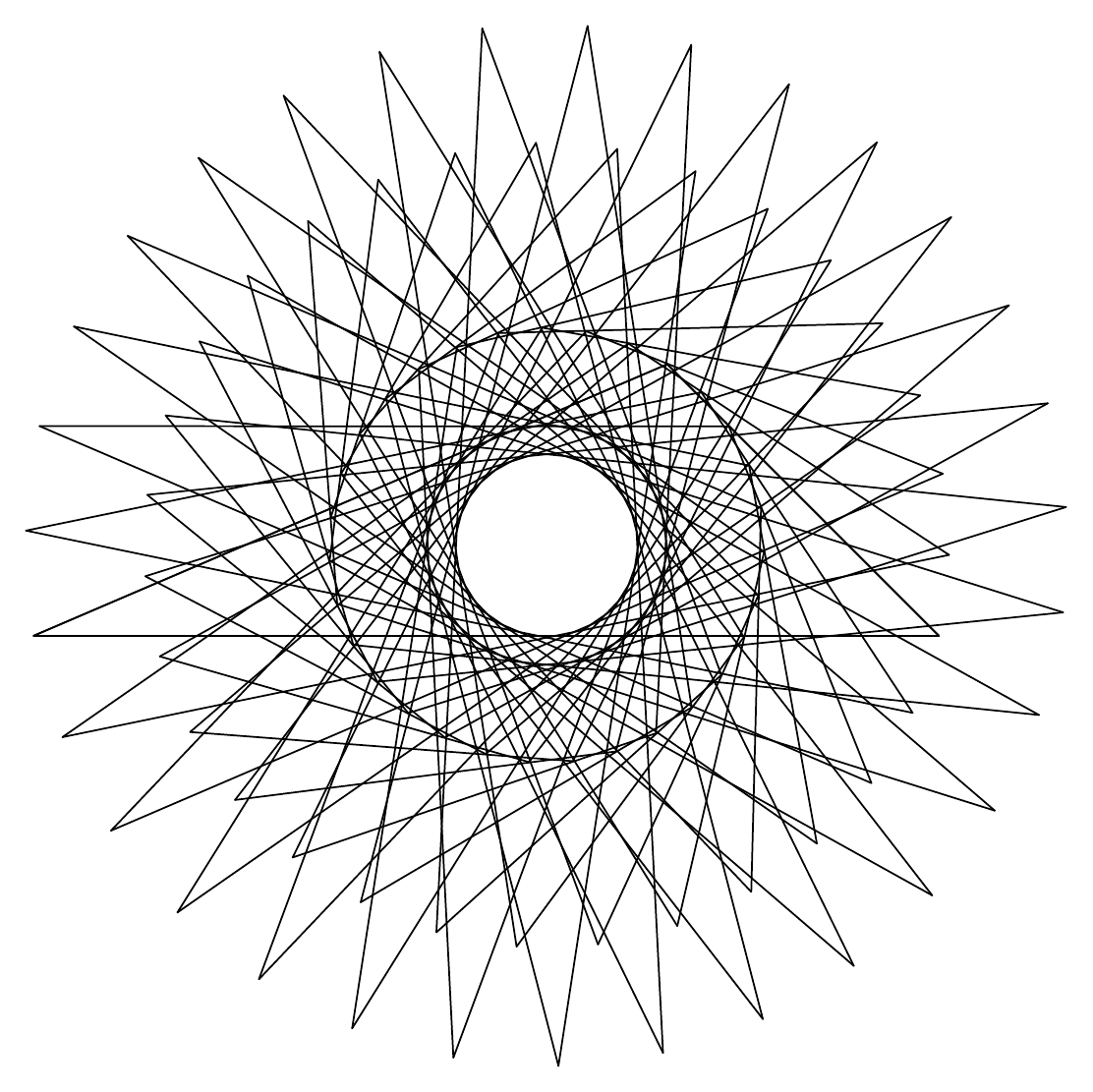} \\
 \text{$\{\cSap{0,\frac{1}{4},\frac{1}{4}}^n(\Delta)\}_{n\in\Z} $} & 
 \text{$\{\cSap{0,\frac{1}{31},\frac{1}{31}}^n(\Delta)\}_{n\in\Z} $} 
\end{tabular}
%\caption{Simple rotation : Orbits of $\cSap{0,\theta,\theta}$}
\end{center}
\end{figure}
\end{Example}

%\pagebreak
\begin{Example} 
Trefoil-looking figure appears in the orbits of $\cSap{\theta,0,-\theta}$.
Starting from the initial triangle $\Delta=(0,\,1,\,0.7+0.3\,i)$, we have
the following two sample illustrations:
\begin{figure}[h]%[htbp]
\begin{center}
\begin{tabular}{cc}
 \ig{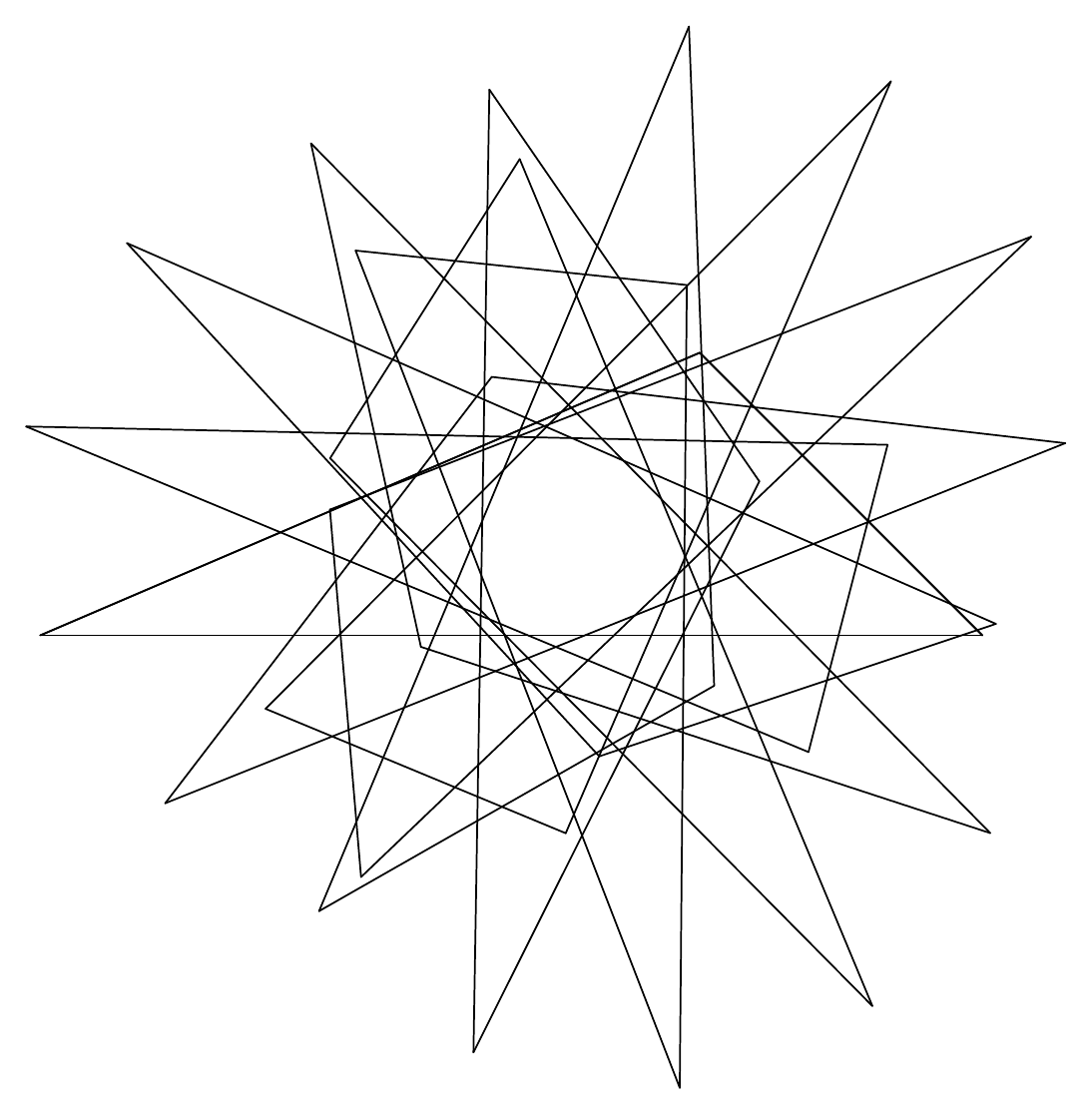} & 
 \ig{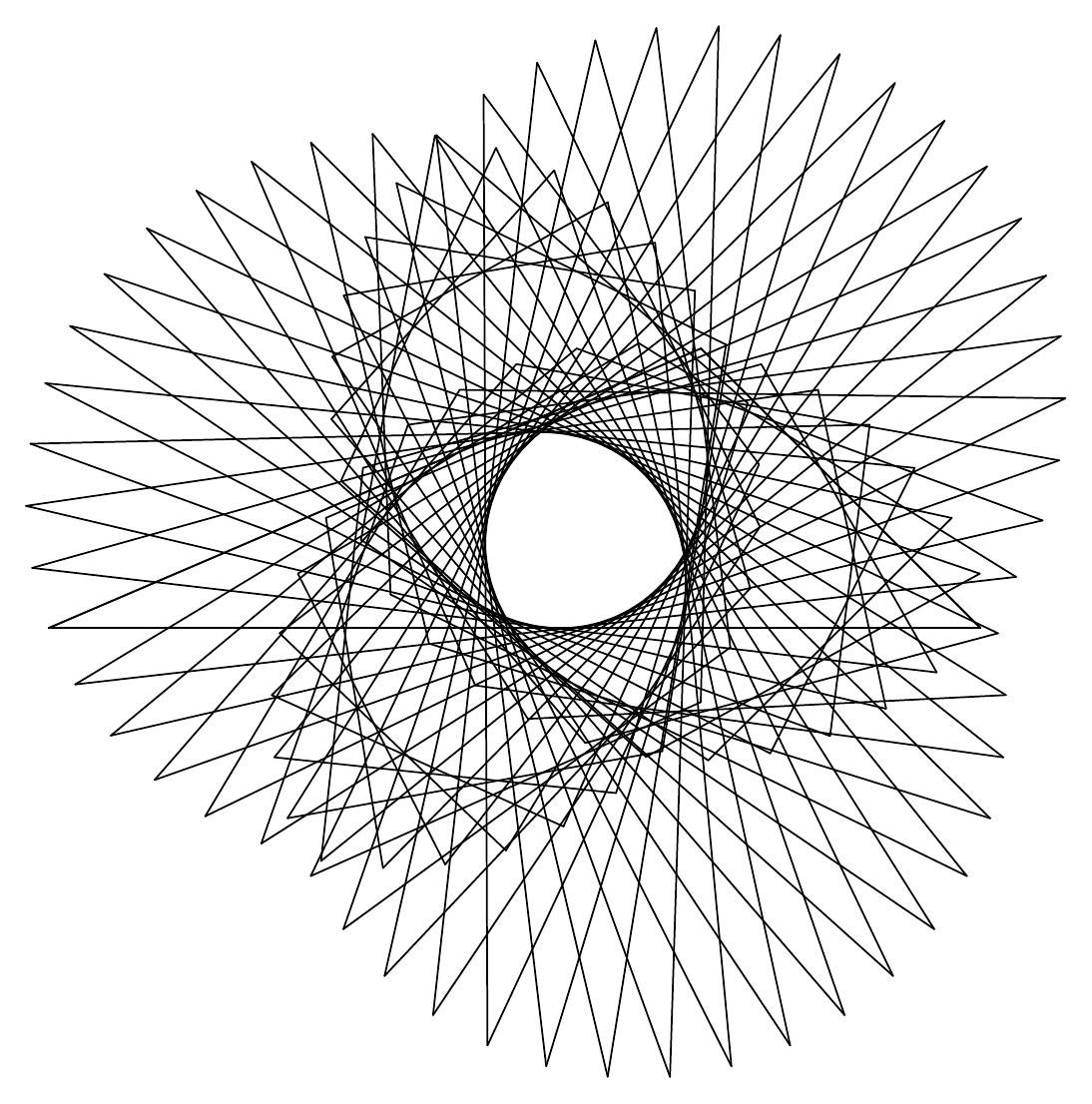} \\
 \text{$\{\cSap{\frac{1}{11},0,-\frac{1}{11}}^n(\Delta)\}_{n\in\Z} $} & 
 \text{$\{\cSap{\frac{1}{37},0,-\frac{1}{37}}^n(\Delta)\}_{n\in\Z} $} 
\end{tabular}
%\caption{Trefoil : Orbits of $\cSap{\theta,0,-\theta}$}
\end{center}
\end{figure}
\end{Example}

\pagebreak
\begin{Example}  
Here are two orbits of $\cSap{\theta,\theta,0}$-images starting from 
$\Delta=(0,\,1,\,0.7+0.5\,i)$
in the cases $\theta=\frac15,\frac1{31}$.

\begin{table}[h]%[htbp]
\begin{center}
\begin{tabular}{cc}
 \ig{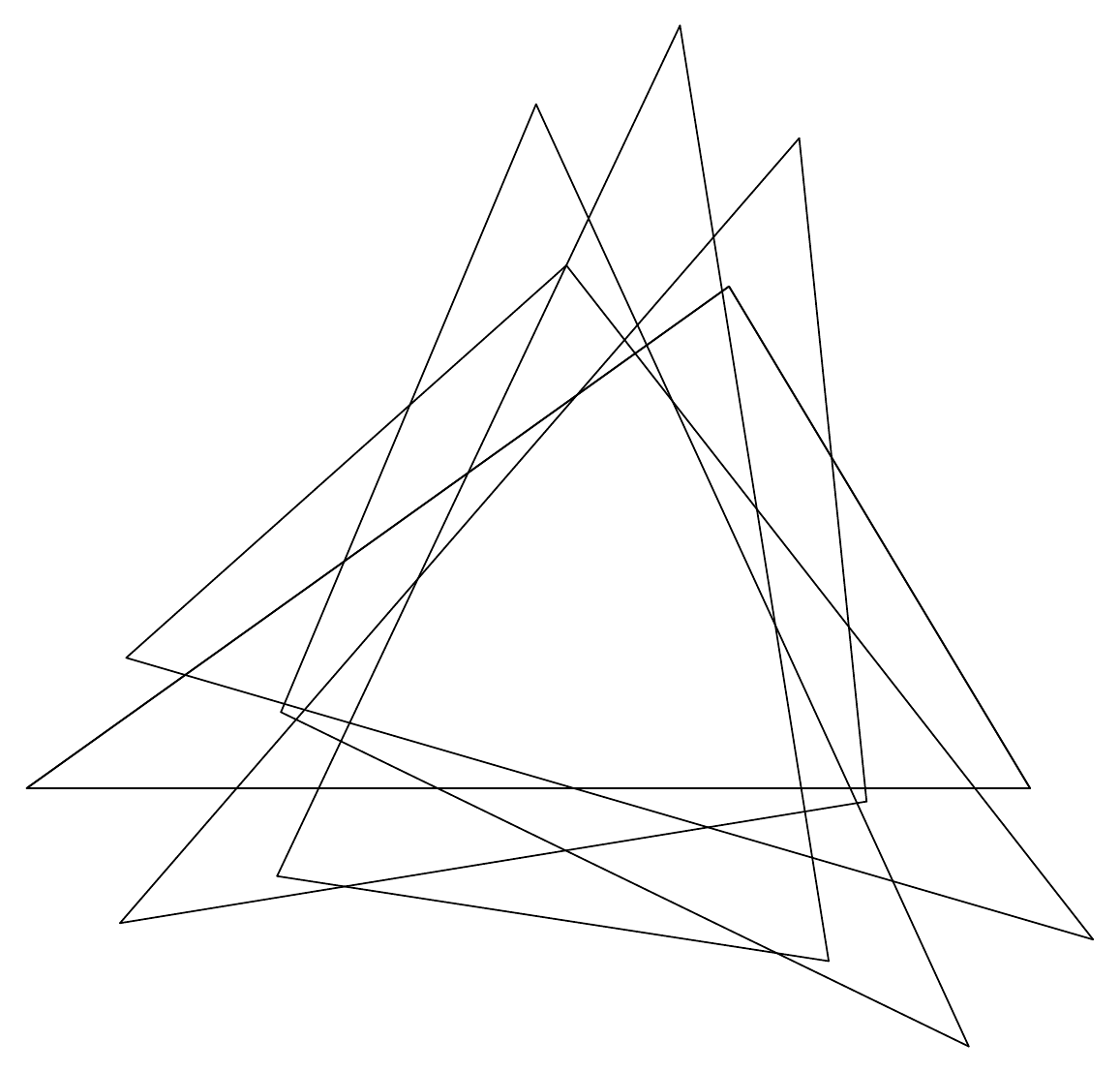} & 
 \ig{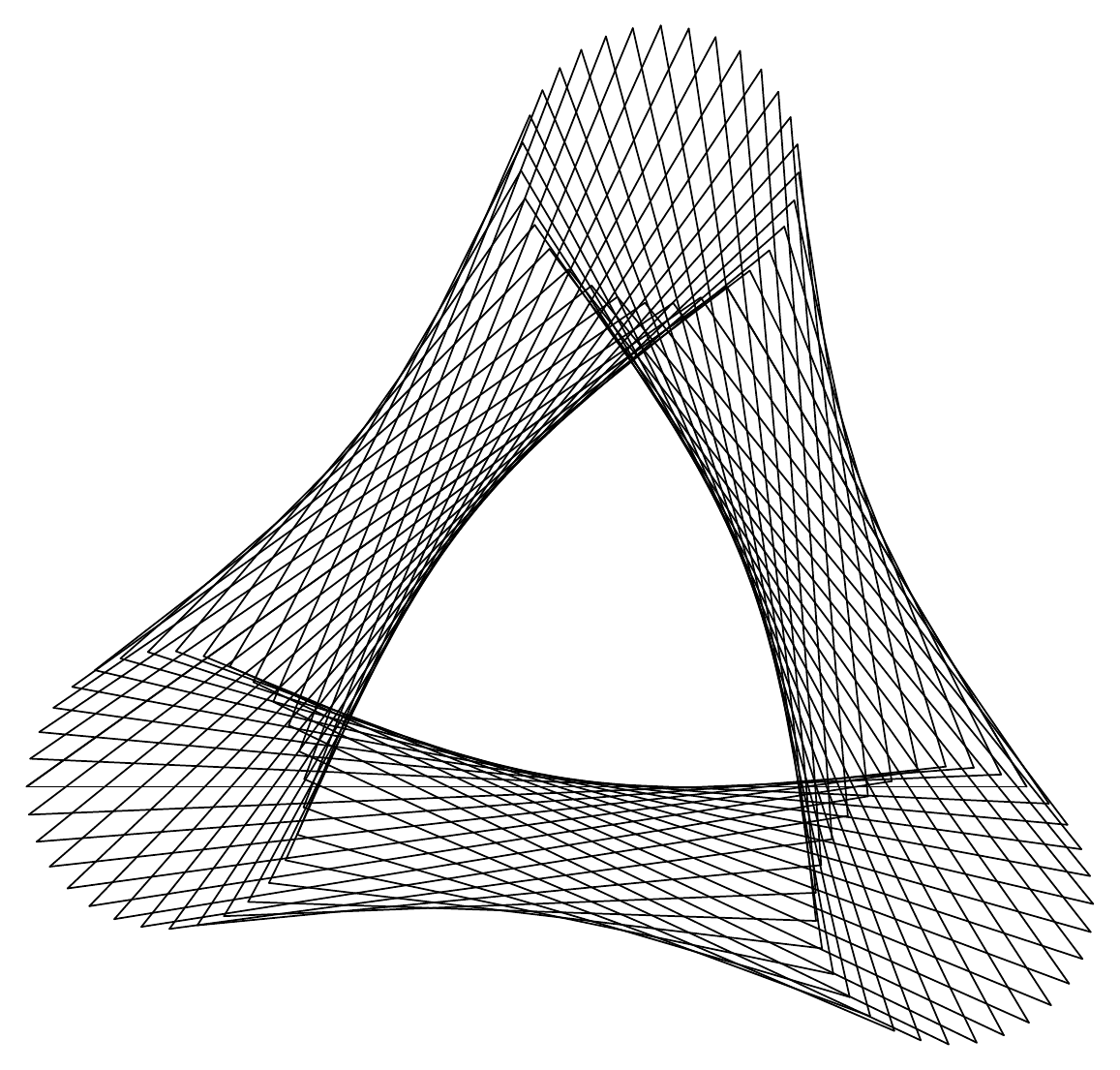} \\
 \text{\{$\cSap{\frac{1}{5},\frac{1}{5},0}^n(\Delta)\}_{n\in\Z} $} &
 \text{\{$\cSap{\frac{1}{31},\frac{1}{31},0}^n(\Delta)\}_{n\in\Z} $} 
\end{tabular}
%\caption{Orbits of $\cSap{\theta,\theta,0}$}
\end{center}
\end{table}
\end{Example}

%\newpage
\ifx\undefined\bysame
\newcommand{\bysame}{\leavevmode\hbox to3em{\hrulefill}\,}
\fi

\end{document}